\documentclass[aap,MSNbibl,dvips]{arximspdf}
\usepackage{graphicx}
\doi{10.1214/12-AAP907} 
\volume{23}
\issue{6}
\pubyear{2013}
\firstpage{2458}
\lastpage{2471}

\makeatletter
\newcommand{\rrvert}{\vert}
\newcommand{\llvert}{\vert}

\newcommand{\R}{\mathbb R}
\newcommand{\Z}{\mathbb Z}
\newcommand{\PR}{\mathbb P}
\newcommand{\E}{\mathbb E}

\newtheorem{theorem}{Theorem}[section]
\newtheorem{corollary}[theorem]{Corollary}

\newproclaim{definition}[theorem]{Definition}

\newtheorem{lemma}[theorem]{Lemma}
\newtheorem{proposition}[theorem]{Proposition}

\newproclaim{remark}{Remark}

\makeatother

\begin{document}
\begin{frontmatter}

\title{Phase transitions in exponential random graphs}
\runtitle{Phase transitions in random graphs}

\begin{aug}
\author[A]{\fnms{Charles} \snm{Radin}\thanksref{t2}\ead[label=e1]{radin@math.utexas.edu}\ead[label=u1,url]{http://ma.utexas.edu/users/radin}}
\and
\author[A]{\fnms{Mei} \snm{Yin}\corref{}\ead[label=e2]{myin@math.utexas.edu}\ead[label=u2,url]{http://ma.utexas.edu/users/myin}}
\runauthor{C. Radin and M. Yin}
\affiliation{University of Texas at Austin}
\address[A]{Department of Mathematics\\
University of Texas at Austin\\
Austin, Texas 78712\\
USA\\
\printead{e1}\\
\hphantom{E-mail: }\printead*{e2}\\
\printead{u1}\\
\hphantom{URL: }\printead*{u2}} 
\end{aug}

\thankstext{t2}{Supported by NSF Grant DMS-07-00120.}

\received{\smonth{9} \syear{2011}}
\revised{\smonth{11} \syear{2012}}

%
\begin{abstract}
We derive the full phase diagram for a large family of two-parameter
exponential random graph models, each containing a first order
transition curve ending in a critical point.
\end{abstract}

%
\begin{keyword}[class=AMS]
\kwd[Primary ]{05C80}
\kwd[; secondary ]{82B26}
\end{keyword}
\begin{keyword}
\kwd{Graph limits}
\kwd{phase transitions}
\kwd{dense random graphs}
\end{keyword}

\end{frontmatter}

\section{Introduction}
We will treat a class of models of large,
``dense'' random graphs, that is, simple graphs on $n$ vertices in
which the
average number of edges is of order $n^2$. More specifically we will
consider two-parameter families of exponential random graphs in which
dependence between the
random edges is defined through certain finite subgraphs $H_2$, in
imitation of
the use of potential energy to provide dependence between
particle states in a grand canonical ensemble
of statistical physics. Intuitively, the two parameters allow one to
adjust the density of edges and the density of subgraphs $H_2$, and
analyze the extent to which specific values of these densities ``interfere''
with one another.
Exponential random graphs have been widely
studied (see \cite{F1,F2} for a range of recent work) since
the pioneering work on the independent case by Erd\H{o}s and R\'enyi
\cite{ER}.
We will concentrate on the phenomenon of phase transitions
which can emerge for dependent variables: a sharp, unambiguous
partition of
parameter ranges separating those values in which
changes in parameters lead to smooth
changes in the two densities, from those special parameter values
where the
response in the densities is singular. This subject has attracted
enormous interest in mathematics, as well as in various
applied disciplines
(some references may be found, e.g.,
in H\"{a}ggstr\"{o}m and Jonasson~\cite{HJ}). Analyses using
mean-field and other uncontrolled approximations (see, e.g.,
\mbox{\cite{PN1,PN2}}) have predicted such partitions, but with some
qualitative error which we discuss in the last section.
There has
recently been important progress by Chatterjee and Diaconis \cite{CD},
including the first rigorous proof of singular dependence on
parameters. We will extend
their result both in the class of models and parameter values under
control and provide an appropriate formalism of phase structure for
such models.

\section{Statement of results}
We consider the class of models in which the probability of the simple graph
$G_n$ on $n$ vertices is given by
%
\begin{equation}
\PR_n^{\beta_1,\beta_2}(G_n)=e^{n^2[\beta_1 t(H_1,G_n)+\beta_2
t(H_2,G_n)-\psi_n]},
\end{equation}
where $H_1$ is an edge, $H_2$ is any finite simple graph with $p\ge
2$ edges, $\psi_n=\psi_n(\beta_1,\beta_2)$ is the normalization
constant, $t(H,G_n)$ is
the density of graph homomorphisms $H \to G_n$
%
\begin{equation}
t(H,G_n)=\frac{|{\operatorname{hom}(H,G_n)}|}{|V(G_n)|^{|V(H)|}}
\end{equation}
and $V(\cdot)$ denotes the vertex set.
Expectation of a real function
of a random graph is denoted $\E_{\beta_1,\beta_2}\{\cdot\}$.
Our main results are the following.
%
\begin{theorem}
\label{One}
For any allowed $H_2$, the pointwise limit
%
\begin{equation}
\psi_\infty(\beta_1,\beta_2)=\lim
_{n\to
\infty}\psi_n(\beta_1,
\beta_2)
\end{equation}
exists and is analytic at all $(\beta_1,\beta_2)$ in the upper
half-plane ($\beta_2\geq0$) except on
a certain continuous curve $\beta_2=q(\beta_1)$ which includes the endpoint
%
\begin{equation}
\bigl(\beta_1^c, \beta_2^c
\bigr)= \biggl(\frac{1}{2}\log(p-1)-\frac
{p}{2(p-1)}, \frac{p^{p-1}}{2(p-1)^p}
\biggr).
\end{equation}
The derivatives
$ \frac{\partial}{\partial
\beta_1}\psi_\infty$ and $ \frac{\partial}{\partial
\beta_2}\psi_\infty$
have (jump) discontinuities across the curve, except at
$(\beta_1^c, \beta_2^c)$ where, however, all the second derivatives
$ \frac{\partial^2}{\partial
\beta_1^2}\psi_\infty$, $ \frac{\partial^2}{\partial
\beta_1 \partial\beta_2}\psi_\infty$ and $ \frac{\partial
^2}{\partial
\beta_2^2}\psi_\infty$ diverge.
\end{theorem}
%
\begin{theorem}
\label{Two} If the graph $H_2$ is a
$p$-star, $p\ge2$, then the pointwise limit
$\psi_\infty(\beta_1,\beta_2)$ exists and is analytic at all
$(\beta_1,\beta_2)$ in the lower half-plane ($\beta_2\leq0$).
\end{theorem}
\begin{remark*}
A $p$-star has $p$ edges meeting at a vertex.
\end{remark*}
%
\begin{corollary}
\label{Corr} For any allowed $H_2$,
the parameter space $\{(\beta_1,\beta_2)\dvtx\break  \beta_2\geq0\}$
consists of a single phase with a
first order phase transition across the indicated curve and a
second order phase transition at the critical point $(\beta_1^c, \beta_2^c)$.
\end{corollary}

To explain the language of phase transitions in Corollary \ref{Corr} we
first give a superficial introduction to the formalism of classical
statistical mechanics within $d$-dimensional lattice gas models; for
more details see, for instance,~\cite{M}.\eject

Assume each point in a $d$-dimensional cube
%
\begin{equation}
C=\bigl\{-n,-(n-1)\cdots0,\cdots(n-1),n\bigr\}^d\subset
\Z^d
\end{equation}
is randomly occupied (by one particle) or not occupied, and assume
there is a (many-body) potential energy of fixed value $a\ne0$
associated with
every occupied subset of $C$ congruent to a certain $H_2\subset C$.
The interaction is attractive if $a<0$ and repulsive if
$a>0$. We define the probability that the occupied sites in $C$
are precisely $c$ by
%
\begin{equation}
\PR_n^{\beta,\mu}(c)=\frac{e^{-\beta[\mu E_1(c) + a E_2(c)]}} {
\Z_n},
\end{equation}
where the parameter $\beta>0$ is called the inverse
temperature, the parameter $\mu\in\R$ is called the chemical potential,
the normalization constant $\Z_n(\beta,\mu)$ is called the partition function,
$H_1$ and $H_2$ are subsets of $C$ with $H_1$ a
singleton and the cardinality $|H_2|\ge2$, and $E_j(c)$ is the
number of copies of $H_j$ in $c$. (We are using ``free'' boundary
conditions.) One of the basic features of the formalism is that the
free energy density, $F_n(\beta,\mu)=\ln[\Z_n(\beta,\mu)]/n^d$,
contains all ways to interact with or influence the system, so that
``all'' physically significant quantities can be obtained by
differentiating it with respect to $\beta$ and $\mu$. For instance,
%
\begin{equation}
\frac{\partial}{\partial\mu}F_n(\beta,\mu)=-{\beta} \E_{\beta,\mu}
\biggl
\{\frac{E_1}{n^d} \biggr\},
\end{equation}
the (average) particle density.
To model materials in thermal equilibrium, calculations in this
formalism normally require that the system size be sufficiently
large, and in practice one often resorts to using $n\to\infty$.
With this as motivation we tentatively define a ``phase'' as a
set of states (i.e., probability distributions) corresponding to a
connected region of the $(\beta,\mu)$ parameter space, which is
maximal for the condition that $ \lim_{n\to
\infty}\frac{\partial^{j+k}}{\partial\beta^j\,\partial
\mu^k}F_n(\beta,\mu)$ are\vspace*{1pt} analytic in $\beta$ and $\mu$ for all $j,k$.
One associates a ``phase transition'' with singularities
which develop in some of these quantities as the system size
diverges. Such singularities are sometimes detected in simulations
through the variances
$\sigma^2_1(n), \sigma^2_2(n)$ of the
particle and energy densities, $E_1/n^d, E_2/n^d$, respectively.
It is easy to check, for instance, that
%
\begin{equation}
\frac{\partial^{2}}{\partial\mu^2}F_n(\beta,\mu) = \beta^2n^d
\sigma^2_1(n),
\end{equation}
and therefore\vspace*{-1pt} a finite limit for $
\frac{\partial^{2}}{\partial\mu^2}F_n(\beta,\mu)$ as $n\to\infty$
implies $\sigma^2_1(n)\to0$ at least as fast as $1/n^d$, while the
divergence of $ \frac{\partial^{2}}{\partial
\mu^2}F_n(\beta,\mu)$ more slowly than $n^d$ as $n\to\infty$
implies
$\sigma^2_1(n)\to0$ more slowly than $1/n^d$,
and a jump discontinuity in $ \frac{\partial}{\partial\mu}
F_n(\beta,\mu)$ as $n\to\infty$ implies\vspace*{2pt} $\sigma^2_1(n)$ does not go
to 0
as $n\to\infty$.
An important simplification was proven by Yang and Lee \cite{YL} who
showed that the limiting free energy density
$ F_\infty(\beta,\mu)=\lim_{n\to\infty} F_n(\beta,\mu)$
always exists and that certain limits commute:
%
\begin{equation}
\label{YL}\quad \lim_{n\to
\infty}\frac{\partial^{j+k}}{\partial\beta^j\,\partial
\mu^k}F_n(\beta,
\mu)= \frac{\partial^{j+k}}{\partial\beta^j\,\partial
\mu^k}\lim_{n\to
\infty}F_n(\beta,\mu)=
\frac{\partial^{j+k}}{\partial\beta^j\,\partial
\mu^k}F_\infty(\beta,\mu).
\end{equation}
This implies that phases and phase transitions can
be determined from the limiting free energy density, and so a
phase is commonly defined (see, e.g., \cite{FR}) as a
connected region of the
$(\beta,\mu)$ parameter
space maximal for the condition that $F_{\infty}(\beta,\mu)$ is
analytic.

Using the obvious analogues for random graphs, with $\beta_1$
playing the role of $-\beta\mu$ and $\beta_2$ the role of $-\beta
a$ (and therefore positive if and only if the model is
``attractive''), $\psi_n=\psi_n(\beta_1,\beta_2)$ plays the key role
of the free
energy density $F_n(\beta,\mu)$. We will show in Theorems \ref{A} and
\ref{B} below
that the limiting free energy density $\psi_{\infty}(\beta_1, \beta
_2)$ exists,
and the proof of Theorem 2 by Yang and Lee \cite{YL}, on the
commutation of limits, then goes through without any difficulty in this
setting, so we can again define phases and phase transitions through
the limiting free energy density, as follows.
%
\begin{definition}
A phase is a connected region of the parameter space
$\{(\beta_1,\beta_2)\}$, maximal for the condition that
the limiting free energy density, $\psi_{\infty}(\beta_1, \beta_2)$,
is analytic. There is a $j$th-order transition at
a boundary point of a phase if at least one
$j$th-order partial derivative of $\psi_{\infty}(\beta_1,
\beta_2)$ is discontinuous there, while all lower order
derivatives are continuous.
\end{definition}
Theorems \ref{One} and \ref{Two} thereby justify our
interpretation in Corollary \ref{Corr} that each of our models
consists of a
single phase with a first order phase transition across the
indicated curve, except at the end (or ``critical'') point $(\beta
^c_1,\beta^c_2)$, where
the transition is second order, superficially similar to the
transition between
liquid and gas in equilibrium materials.

\section{Proofs}
Chatterjee and
Diaconis have proven that the main object of interest,
$\psi_\infty(\beta_1,\beta_2)$, exists for all $\beta_1$ and
nonnegative $\beta_2$ and is the solution of a certain
optimization problem:
%
\begin{theorem}[(Part of Theorem 4.1 in \cite{CD})]
\label{CD} Fix one of our models and assume $H_2$ has $p\ge2$
edges. Then for all $(\beta_1,\beta_2)$ in the upper half-plane $(\beta_2>0)$,
the pointwise limit
$ \psi_\infty(\beta_1,\beta_2)=\lim_{n\to
\infty}\psi_n(\beta_1,\beta_2)$ exists and
%
\begin{equation}
\label{fe}\qquad \psi_\infty(\beta_1,\beta_2)=\sup
_{u\in[0,1]} \biggl(\beta_1u+\beta_2u^p-
\frac{1}{2}u\log u-\frac{1}{2}(1-u)\log(1-u) \biggr).
\end{equation}
\end{theorem}
The following detailed analysis of this maximization problem is therefore
fundamental to understanding the phase structure of our models, and we
now address it.
%
\begin{proposition}
\label{max} Fix an integer $p\geq2$. Consider the maximization
problem for
%
\begin{equation}
\label{l} l(u; \beta_1, \beta_2)=\beta_1u+
\beta_2u^p-\tfrac{1}{2}u\log u-\tfrac{1}{2}(1-u)
\log(1-u)
\end{equation}
on the interval $[0,1]$, where $-\infty<\beta_1<\infty$ and
$-\infty<\beta_2<\infty$ are parameters. Then there is a $V$-shaped
region in the $(\beta_1, \beta_2)$ plane with corner point
$(\beta_1^c, \beta_2^c)$ such that outside this region, $l(u)$ has
a unique local maximizer (hence global maximizer) $u^*$; whereas
inside this region, $l(u)$ always has exactly two local maximizers
$u_1^*$ and $u_2^*$. Moreover, for every $\beta_1$ inside this
$V$-shaped region ($\beta_1<\beta_1^c$), there is a unique
$\beta_2=q(\beta_1)$ such that the two local maximizers of $l(u;
\beta_1, q(\beta_1))$ are both global maximizers. Furthermore $q$
is a continuous and decreasing function of $\beta_1$.
\end{proposition}
\begin{remark*}
By the Lebesgue Differentiation theorem, $q$ being monotone guarantees
that it is differentiable almost everywhere.
\end{remark*}
\begin{pf*}{Proof of Proposition \ref{max}}
The location of maximizers of $l(u)$ on the interval $[0,1]$ are
closely related to properties of its derivatives $l'(u)$ and $l''(u)$,
%
\begin{eqnarray}
\label{lone} l'(u)&=&\beta_1+p
\beta_2u^{p-1}-\frac{1}{2}\log\frac{u}{1-u},
\\
%
\label{ltwo} l''(u)&=&p(p-1)
\beta_2u^{p-2}-\frac{1}{2u(1-u)}.
\end{eqnarray}

We first analyze properties of $l''(u)$ on the interval $[0,1]$.
Consider instead the function
%
\begin{equation}
\label{beta2} m(u)=\frac{1}{2p(p-1)u^{p-1}(1-u)}.
\end{equation}
Simple optimization shows
%
\begin{equation}
0\leq(p-1)u^{p-1}(1-u)\leq\biggl(\frac{p-1}{p}
\biggr)^p,
\end{equation}
and the equality holds if and only if $u=\frac{p-1}{p}$. Thus
%
\begin{equation}
\frac{p^{p-1}}{2(p-1)^p} \leq m(u)<\infty
\end{equation}
with the lower bound achieved only at $u=\frac{p-1}{p}$, and
$m(u)$ decreases before this minimum and increases after it. This
implies that for $\beta_2 \leq\frac{p^{p-1}}{2(p-1)^p}$,
$l''(u)\leq0$ on the whole interval $[0,1]$; whereas for
$\beta_2>\frac{p^{p-1}}{2(p-1)^p}$, $l''(u)$ will take on both
positive and negative values, and we denote the transition points
by $u_1$ and $u_2$ ($u_1<\frac{p-1}{p}<u_2$).\eject

Based on properties of $l''(u)$, we next analyze properties of
$l'(u)$ on the interval $[0,1]$. For $\beta_2 \leq
\frac{p^{p-1}}{2(p-1)^p}$, $l'(u)$ is monotonically decreasing.
For $\beta_2>\frac{p^{p-1}}{2(p-1)^p}$, $l'(u)$ is decreasing from
$0$ to $u_1$, increasing from $u_1$ to $u_2$ and then decreasing
again from $u_2$ to $1$.

Based on properties of $l'(u)$ and $l''(u)$, we analyze properties
of $l(u)$ on the interval $[0,1]$. Independent of the choice of
parameters $\beta_1$ and $\beta_2$, $l(u)$ is a bounded continuous
function, $l'(0)=\infty$ and $l'(1)=-\infty$, so $l(u)$ cannot
be maximized at $0$ or $1$. For $\beta_2 \leq
\frac{p^{p-1}}{2(p-1)^p}$, $l'(u)$ crosses the $u$-axis only once,
going from positive to negative. Thus $l(u)$ has a unique local
maximizer (hence global maximizer) $u^*$. For $\beta_2>
\frac{p^{p-1}}{2(p-1)^p}$, the situation is more complicated and
deserves a careful analysis. If $l'(u_1)\geq0$ [resp.,
$l'(u_2)\leq0$], $l(u)$ has a unique local maximizer (hence
global maximizer) at a point $u^*>u_2$ (resp., $u^*<u_1$). If
$l'(u_1)<0<l'(u_2)$, then $l(u)$ has two local maximizers $u_1^*$
and $u_2^*$, with $u_1^*<u_1<\frac{p-1}{p}<u_2<u_2^*$.\vspace*{1pt}

Notice that $u_1$ and $u_2$ are solely determined by the choice of
parameter $\beta_2>\frac{p^{p-1}}{2(p-1)^p}$, and vice versa. By
(\ref{beta2}),
%
\begin{eqnarray}
l'(u_1)&=&\beta_1+\frac{1}{2(p-1)(1-u_1)}-
\frac{1}{2}\log\frac{u_1}{1-u_1},
\\
%
l'(u_2)&=&\beta_1+\frac{1}{2(p-1)(1-u_2)}-
\frac{1}{2}\log\frac{u_2}{1-u_2}.
\end{eqnarray}
Consider the function
%
\begin{equation}
n(u)=\frac{1}{2(p-1)(1-u)}-\frac{1}{2}\log\frac{u}{1-u}.
\end{equation}
It is not hard to see that $n(0)=\infty$, $n(1)=\infty$, $n(u)$ is
decreasing from $0$ to $\frac{p-1}{p}$, then increasing from
$\frac{p-1}{p}$ to $1$, and the global minimum value is
%
\begin{equation}
n \biggl(\frac{p-1}{p} \biggr)=\frac{p}{2(p-1)}-\frac{1}{2}
\log(p-1).
\end{equation}
This implies in particular that $l'(u_1)\geq0$ for $\beta_1\geq
\frac{1}{2}\log(p-1)-\frac{p}{2(p-1)}$. The only possible region
in the $(\beta_1, \beta_2)$ plane where $l'(u_1)<0<l'(u_2)$ is
thus bounded by $\beta_1<\frac{1}{2}\log(p-1)-\frac{p}{2(p-1)}$
and $\beta_2>\frac{p^{p-1}}{2(p-1)^p}$.

\begin{figure}

\includegraphics{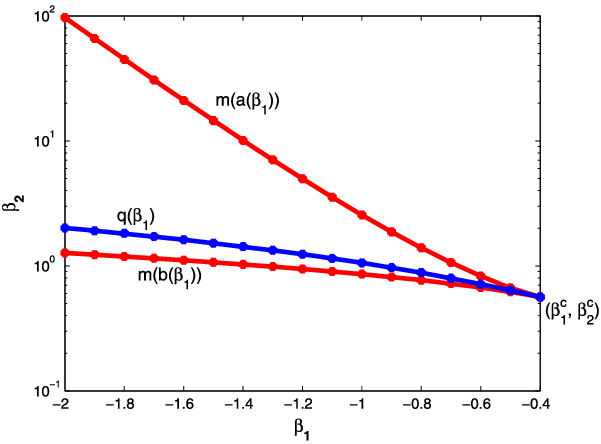}\vspace*{-3pt}

\caption{The $V$-shaped region (with phase transition curve inside)
in the $(\beta_1, \beta_2)$ plane. Graph drawn for $p=3$.}
\label{V-shape}\vspace*{-5pt}
\end{figure}

\begin{figure}[b]\vspace*{-3pt}

\includegraphics{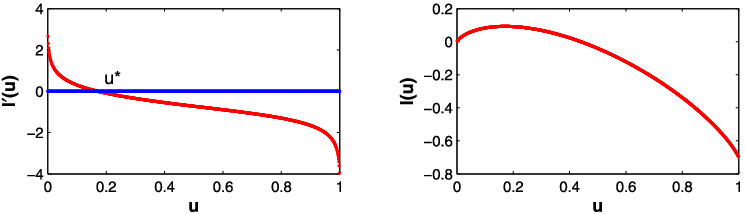}\vspace*{-3pt}

\caption{Outside the $V$-shaped region, $l(u)$ has
a unique local maximizer (hence global maximizer)~$u^*$. Graph
drawn for $\beta_1=-0.8$, $\beta_2=0.1$ and $p=3$.}
\label{below}
\end{figure}

We now analyze the behavior of $l'(u_1)$ and $l'(u_2)$ more
closely when $\beta_1$ and $\beta_2$ are chosen from this region.
Recall that by construction, $u_1<\frac{p-1}{p}<u_2$. By
monotonicity of $n(u)$ on the intervals $(0, \frac{p-1}{p})$ and
$(\frac{p-1}{p}, 1)$, there exist continuous functions
$a(\beta_1)$ and $b(\beta_1)$ of $\beta_1$, such that $l'(u_1)<0$
for $u_1>a(\beta_1)$ and $l'(u_2)>0$ for $u_2>b(\beta_1)$.
$a(\beta_1)$ is an increasing function of $\beta_1$, whereas
$b(\beta_1)$ is a decreasing function, and they satisfy
%
\begin{equation}
n\bigl(a(\beta_1)\bigr)=n\bigl(b(\beta_1)\bigr)=-
\beta_1.
\end{equation}
Also, as $\beta_1\rightarrow-\infty$, $a(\beta_1)\rightarrow0$
and $b(\beta_1)\rightarrow1$. By (\ref{beta2}), the restrictions
on $u_1$ and $u_2$ yield restrictions on $\beta_2$. We have
$l'(u_1)<0$ for $\beta_2<m(a(\beta_1))$ and $l'(u_2)>0$ for
$\beta_2>m(b(\beta_1))$. Notice that $m(a(\beta_1))$ and
$m(b(\beta_1))$ are both decreasing functions of $\beta_1$, and as
$\beta_1\rightarrow-\infty$, they both grow unbounded. By
construction, for every parameter value $(\beta_1, \beta_2)$,
$l'(u_2)>l'(u_1)$. Also, for fixed $\beta_1$, $m(a(\beta_1))$ is
the value of $\beta_2$ for which $l'(u_1)=0$, and $m(b(\beta_1))$
is the value for which $l'(u_2)=0$. Thus the curve $m(b(\beta_1))$
must lie below the curve $m(a(\beta_1))$. And together they
generate the bounding curves of the $V$-shaped region in the
$(\beta_1, \beta_2)$ plane where two local maximizers exist for
$l(u)$. It is not hard to see that the corner point is given by
$(\beta_1^c, \beta_2^c)= (\frac{1}{2}\log
(p-1)-\frac{p}{2(p-1)}, \frac{p^{p-1}}{2(p-1)^p} )$. (See
Figures \ref{V-shape}--\ref{qcurve}.)

Fixing an arbitrary $\beta_1<\beta_1^c$, we examine the effect of
varying $\beta_2$ on the graph of $l'(u)$. It is clear from
(\ref{lone}) that $l'(u)$ shifts upward as $\beta_2$ increases. As a
result, as $\beta_2$ gets large, the positive area bounded by the
curve $l'(u)$ increases, whereas the negative area decreases. By
the fundamental theorem of calculus, the difference between the
positive and negative areas is the difference between $l(u_2^*)$
and $l(u_1^*)$, which goes from negative [$l'(u_2)=0$, $u_1^*$ is
the global maximizer] to positive [$l'(u_1)=0$, $u_2^*$ is the
global maximizer] as $\beta_2$ goes from $m(b(\beta_1))$ to
$m(a(\beta_1))$. Thus there must be a unique $\beta_2\dvtx
m(b(\beta_1))<\beta_2<m(a(\beta_1))$ such that $u_1^*$ and $u_2^*$
are both global maximizers. We denote this $\beta_2$ by
$q(\beta_1)$; see Figures \ref{V-shape} and \ref{qcurve}.

\begin{figure}[t]

\includegraphics{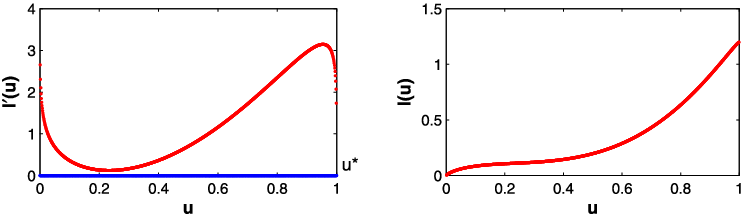}

\caption{Outside the $V$-shaped region, $l(u)$ has
a unique local maximizer (hence global maximizer) $u^*$. Graph
drawn for $\beta_1=-0.8$, $\beta_2=2$ and $p=3$.}
\label{above}
\end{figure}

\begin{figure}[b]

\includegraphics{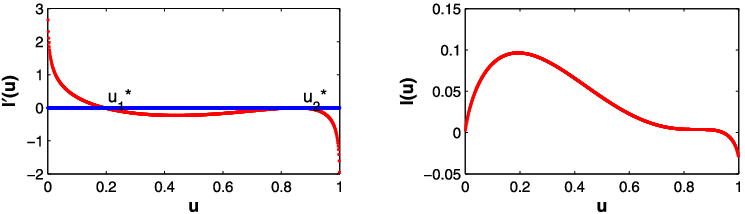}

\caption{Along the lower bounding curve of the $V$-shaped region, $l'(u)$ has
two zeros $u_1^*$ and $u_2^*$, but only $u_1^*$ is the global
maximizer for $l(u)$. Graph drawn for $\beta_1=-0.8$,
$\beta_2=0.769$ and $p=3$.}
\label{lower}
\end{figure}

\begin{figure}

\includegraphics{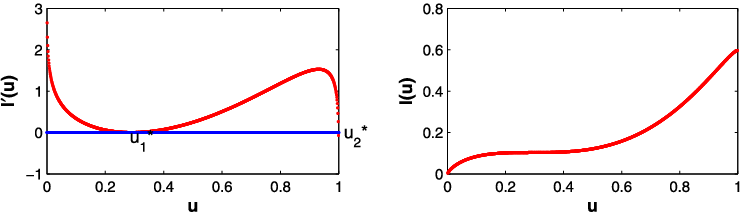}

\caption{Along the upper bounding curve of the $V$-shaped region, $l'(u)$ has
two zeros $u_1^*$ and $u_2^*$, but only $u_2^*$ is the global
maximizer for $l(u)$. Graph drawn for $\beta_1=-0.8$,
$\beta_2=1.396$ and $p=3$.}
\label{upper}
\end{figure}

\begin{figure}[b]

\includegraphics{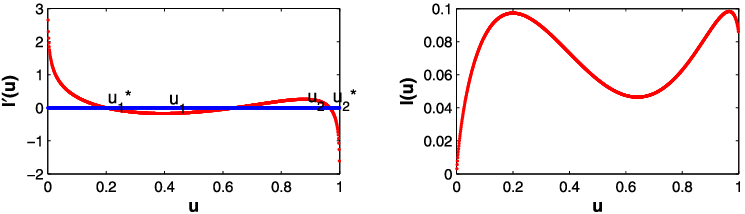}

\caption{Along the phase transition curve, $l(u)$ has
two local maximizers $u_1^*$ and $u_2^*$, and both are global
maximizers. Graph drawn for $\beta_1=-0.8$, $\beta_2=0.884$ and
$p=3$.}
\label{qcurve}
\end{figure}

By analyzing the graph of $l'(u)$, we see that the parameter
values of $(\beta_1, q(\beta_1))$ are exactly the ones for which
positive and negative areas bounded by $l'(u)$ equal each other.
An increase in $\beta_1$ will induce an upward shift of $l'(u)$,
which must be balanced by a decrease in $\beta_2=q(\beta_1)$.
Similarly, a decrease in $\beta_1$ will induce a downward shift of
$l'(u)$, which must be balanced by an increase in
$\beta_2=q(\beta_1)$. This justifies that $q$ is monotonically
decreasing in $\beta_1$. Furthermore, the continuity of $l'(u)$ as
a function of $\beta_1$ and $\beta_2$ implies the continuity of
$q$ as a function of $\beta_1$.
\end{pf*}
%
\begin{corollary}
The transition curve $\beta_2=q(\beta_1)$ displays a universal
asymptotic behavior as $\beta_1 \to-\infty$
%
\begin{equation}
\lim_{\beta_1 \to-\infty}\bigl|q(\beta_1)+\beta_1\bigr|=0.
\end{equation}
\end{corollary}
\begin{pf}
Assume $\beta_1=\beta=-\beta_2+\delta$ where $|\delta|\ne0$, and
define $F(u,\beta)=\beta(u-u^p)$, $I(u)=-(1/2) u\log u -(1/2)
(1-u)\log(1-u)$ and $G(u)=\delta u^p+I(u)$. Then
$\ell(u;\beta_1,\beta_2)=\ell(u,\beta)=F(u,\beta)+G(u)$. We will
show, for sufficiently negative $\beta$, that if $\delta>0$ the
global maximum $u^\ast$ of $\ell(u,\beta)$ equals $u^\ast_2$,
while if $\delta<0$ $u^\ast$ equals $u^\ast_1$. This implies,
for these $\beta$, that $-\beta-|\delta| < q(\beta)<
-\beta+|\delta|$, which will prove the desired limit.

From the continuity of $G(u)$ there exists $\eta\in(0, 1/p)$,
such that $|G(u)-G(u_0)| < |\delta|/2$ for all $|u-u_0|<\eta$, for
both $u_0=0$ and $u_0=1$. Recall that $u^\ast_1<u_1<
\frac{p-1}{p}<u_2<u^\ast_2$. Since $u-u^p=u(1-u^{p-1})>0$ on
$(0,1)$, there exists $B<0$ such that for all $\beta<B$ and $u\in
[\eta, 1-\eta]$, $F(u,\beta)< -1 -|\delta|$ and therefore
$\ell(u,\beta)<0=\ell(0,\beta)$, so $u^\ast\in[0,\eta)\cup
(1-\eta,1]$. Using that $F(u, \beta) \le0$ for all $\beta<0$ and
all $u\in[0,1]$, if $\delta>0$ and $u< \eta<1/p\le(p-1)/p$, then
$\ell(u,\beta)< \delta/2 < \delta=\ell(1,\beta)$ so $u^\ast=
u^\ast_2$, while if $\delta<0$ and $u> 1-\eta>(p-1)/p$, then
$\ell(u,\beta)< -|\delta|/2< 0=\ell(0,\beta)$ and therefore
$u^\ast= u^\ast_1$.
\end{pf}

So far we have used results from \cite{CD} but have avoided specific
reference to the framework of graph limits, developed by Lov\'{a}sz et
al, which was used to prove those results. We now need to refer
directly to graph limits; for the notation and an introduction to this
material see, for instance, \cite{LS} or \cite{CD}.
%
\begin{theorem}
\label{gen} Let $G_n$ be a random graph on $n$ vertices in one of
our models. For parameter values of $(\beta_1, \beta_2)$ in the
upper half-plane $\beta_2>-\frac{2}{p(p-1)}$, the behavior of
$G_n$ in the large $n$ limit is as follows:
%
\begin{equation}
\min_{u\in U}\delta_{\square}(\tilde G_n,\tilde
u)\to0\qquad\mbox{in probability as } n\to\infty,
\end{equation}
where $U$ is the set of maximizers of (\ref{l}).
\end{theorem}
\begin{pf}
The assumptions of Theorems 4.2 or 6.1 in \cite{CD} are satisfied
for parameter values $\beta_2>-\frac{2}{p(p-1)}$. By Proposition
\ref{max}, along the curve $(\beta_1, q(\beta_1))$, the
maximization problem (\ref{l}) is solved at two values $u_1^*$ and
$u_2^*$; whereas off this curve, it is solved at a unique
value $u^*$. Thus in the large $n$ limit, along the curve
$(\beta_1, q(\beta_1))$, $G_n$ behaves like an Erd\H{o}s--R\'{e}nyi
graph $G(n, u)$ ($u$ picked by some distribution from $u_1^*$ and $u_2^*$);
whereas off this curve, $G_n$ is indistinguishable from the
Erd\H{o}s--R\'{e}nyi graph $G(n, u^*)$.
\end{pf}
%
\begin{corollary}
\label{q} Fix any $\beta_2>\beta_2^c$. Let $H$ be an edge,
so $t(H,G_n)$ is the edge
density of $G_n$. Then there exists a continuous and decreasing function
$q^{-1}(\beta_2)$ such that
%
\begin{equation}
\lim_{n\rightarrow\infty}\PR_n^{\beta_1,\beta_2}
\bigl(t(H,G_n)>u_2\bigr)=1 \qquad\mbox{if }
\beta_1>q^{-1}(\beta_2)
\end{equation}
and
%
\begin{equation}
\lim_{n\rightarrow\infty}\PR_n^{\beta_1,\beta_2}
\bigl(t(H,G_n)<u_1\bigr)=1 \qquad\mbox{if }
\beta_1<q^{-1}(\beta_2).
\end{equation}
Here $u_1$ and $u_2$ are defined as in the proof of Proposition
\ref{max}: $m(u_1)=m(u_2)=\beta_2$.
\end{corollary}
\begin{remark*}
As $\beta_2\rightarrow\infty$, $u_1\rightarrow0$ and
$u_2\rightarrow1$ and the jump is noticeable even for relatively
small values of $\beta_2$.
\end{remark*}
\begin{pf*}{Proof of Corollary \ref{q}}
As $q(\beta_1)$ is a continuous and decreasing function of
$\beta_1$, the inverse function $q^{-1}(\beta_2)$ exists and is
also continuous and decreasing. We examine the effect of varying
$\beta_1$ on the graph of $l'(u)$ [and hence on the global
maximizers of $l(u)$]. First note that varying $\beta_1$ does not
change the shape of $l'(u)$. Inside the $V$-shaped region, there are
three cases. Recall that $u_1^*<u_1<\frac{p-1}{p}<u_2<u_2^*$. For
$\beta_1=q^{-1}(\beta_2)$, positive and negative areas bounded by
$l'(u)$ equal each other and thus $u_1^*$ and $u_2^*$ are both global
maximizers. For $\beta_1<q^{-1}(\beta_2)$, the graph of $l'(u)$
shifts downward, negative area exceeds positive area and thus $u_1^*$
is the global maximizer. For $\beta_1>q^{-1}(\beta_2)$, the graph
of $l'(u)$ shifts upward, positive area exceeds negative area and
thus $u_2^*$ is the global maximizer. Outside the $V$-shaped region,
there are two cases. Below the lower bounding curve, $l'(u)$ has a
unique local maximizer $u^*<u_1$. Above the upper bounding curve,
$l'(u)$ has a unique local maximizer $u^*>u_2$. Our conclusion
then follows from Theorem~\ref{gen}.
\end{pf*}
%
\begin{theorem}
\label{p} Assume that in one of our models $H_2$ is a $p$-star
($p\geq2$). For all parameter values $(\beta_1, \beta_2)$, the
behavior of $G_n$ in the large $n$ limit is as follows:
%
\begin{equation}
\min_{u\in U}\delta_{\square}(\tilde G_n,\tilde
u)\to0\qquad\mbox{in probability as } n\to\infty,
\end{equation}
where $U$ is the set of maximizers of (\ref{l}).
\end{theorem}
\begin{pf}
This follows from related results in \cite{CD}. We separate the
parameter plane $\{(\beta_1, \beta_2)\}$ into upper and lower
half-planes. The upper half-plane ($\beta_2\geq0$) satisfies the
assumptions of Theorem 4.2, and the lower half-plane ($\beta_2\leq
0$) satisfies the assumptions of Theorem 6.4. By similar reasoning
as in Theorem \ref{gen}, the rest of the proof follows.
\end{pf}

The following is a straightforward application of the
standard analytic implicit function theorem, which can be found, for
instance, in the text of Krantz and Parks~\cite{KP}, so we omit the proof.
%
\begin{proposition}
\label{ana1} Off the end point $(\beta_1^c, \beta_2^c)$, the
local maximizer $u^*$ for $l(u; \beta_1, \beta_2)$ ($u_1^*$ and
$u_2^*$ if inside the $V$-shaped region) is an analytic function of
the parameters $\beta_1$ and $\beta_2$.
\end{proposition}
%
\begin{proposition}
\label{ana2} Off the phase transition curve, $l(u^*)=\max
l(u; \beta_1, \beta_2)$ [$l(u_1^*)$ or $l(u_2^*)$ if inside the
$V$-shaped region] is an analytic function of the parameters
$\beta_1$ and $\beta_2$.
\end{proposition}
\begin{pf}
It is clear that $l(u; \beta_1, \beta_2)$ is analytic for $u\in
(0,1)$, $\beta_1\in(-\infty, \infty)$, and $\beta_2\in(-\infty,
\infty)$. Outside the $V$-shaped region, $l(u)$ has a unique local
maximizer $u^*$ in $(0,1)$, which is analytic in $\beta_1$ and
$\beta_2$ by Proposition \ref{ana1}. Inside the $V$-shaped region,
$l(u)$ has two local maximizers $u_1^*$ and $u_2^*$, both have
values in $(0,1)$ and are analytic in $\beta_1$ and $\beta_2$ by
Proposition \ref{ana1}. Below the phase transition curve, $\max
l(u)$ is given by $l(u_1^*)$, which coincides with $l(u^*)$ along
the lower bounding curve. Above the phase transition curve, $\max
l(u)$ is given by $l(u_2^*)$, which coincides with $l(u^*)$ along
the upper bounding curve. Our claim follows by realizing that
compositions of analytic functions are analytic as long as the
domains and ranges match up.
\end{pf}
%
\begin{theorem}
\label{A} Let $G_n$ be a random graph on $n$ vertices in one of
our models. The limiting free energy density
$\psi_\infty=\lim_{n\rightarrow\infty}\psi_n$ is an analytic
function of the parameters $\beta_1$ and $\beta_2$ off the phase
transition curve in the upper half-plane
$\beta_2>-\frac{2}{p(p-1)}$.
\end{theorem}
\begin{pf}
The assumptions of Theorems 4.1 and 6.1 in \cite{CD} are satisfied
for parameter values $\beta_2>-\frac{2}{p(p-1)}$. Our claim then
follows from Proposition~\ref{ana2}.
\end{pf}
%
\begin{theorem}
\label{B} Assume that in one of our models $H_2$ is a $p$-star
(\mbox{$p\geq2$}). The limiting free energy density
$\psi_\infty=\lim_{n\rightarrow\infty}\psi_n$ is an analytic
function of the parameters $\beta_1$ and $\beta_2$ off the phase
transition curve.
\end{theorem}
\begin{pf}
The assumptions of Theorems 4.1 or 6.4 in \cite{CD} are
satisfied. Our claim again follows from Proposition \ref{ana2}.
\end{pf}
%
\begin{lemma}[(Lov\'{a}sz--Szegedy \cite{LS})]
\label{LS} Let $U, W\dvtx  [0,1]^2\rightarrow[0,1]$ be two symmetric
integrable functions. Then for every finite simple graph $F$,
%
\begin{equation}
\bigl|t(F, U)-t(F, W)\bigr|\leq\bigl|E(F)\bigr|\cdot\delta_{\square}(\tilde{U}, \tilde{W}).
\end{equation}
\end{lemma}
\begin{pf*}{Proof of Theorems \ref{One} and \ref{Two}}
The stated analyticity is proven in Theorems \ref{A} and \ref{B}, so we
only need to examine the situation along the phase transition curve. We
know from Theorems \ref{gen} and \ref{p} that $\tilde G_n$ converges in
probability to $u^*$, off the curve. By Lemma \ref{LS}, $t(H_1, G_n)$
then converges in probability to $t(H_1, u^*)$. As $t(H_1, G_n)$ is
uniformly bounded in $n$, this implies that
%
\begin{equation}
\E_{\beta_1,\beta_2}\bigl\{\bigl|t(H_1,G_n)-t
\bigl(H_1,u^*\bigr)\bigr|\bigr\}\to0 \qquad\mbox{as }n\to\infty.
\end{equation}
Therefore
%
\begin{eqnarray}
\E_{\beta_1,\beta_2}\bigl\{t(H_1,G_n)\bigr\}&\to&
\E_{\beta_1,\beta_2}\bigl\{t\bigl(H_1,u^*\bigr)\bigr\}
\nonumber\\[-8pt]\\[-8pt]
&=&u^*(\beta_1, \beta_2)=\frac{\partial}{\partial
\beta_1}
\psi_{\infty}(\beta_1, \beta_2) \qquad\mbox{as }n\to
\infty.
\nonumber
\end{eqnarray}
Similarly,
%
\begin{eqnarray}\qquad
\E_{\beta_1,\beta_2}\bigl\{t(H_2,G_n)\bigr\}&\to&
\E_{\beta_1,\beta_2}\bigl\{t\bigl(H_2,u^*\bigr)\bigr\}
\nonumber\\[-8pt]\\[-8pt]
&=& \bigl(u^*(\beta_1, \beta_2) \bigr)^p=
\frac{\partial}{\partial
\beta_2}\psi_{\infty}(\beta_1, \beta_2)
\qquad\mbox{as }n\to\infty.
\nonumber
\end{eqnarray}
By Corollary \ref{q}, these\vspace*{1pt} two first derivatives $ \frac
{\partial}{\partial
\beta_1}\psi_{\infty}$ and $ \frac{\partial}{\partial
\beta_2}\psi_{\infty}$ are discontinuous across the curve (except at
the end
point). Let us now take a closer look at the behavior of $\psi_{\infty
}$ at the critical point.
Recall that $l'(u; \beta_1^c, \beta_2^c)$ is monotonically decreasing
on $[0,1]$,
and the unique zero is achieved at $\frac{p-1}{p}$.
Take any $0<\varepsilon<\frac{1}{p}$.
Set $\delta=\min\{l'(\frac{p-1}{p}-\varepsilon),
-l'(\frac{p-1}{p}+\varepsilon)\}$. Consider $(\beta_1, \beta_2)$ so close
to $(\beta_1^c, \beta_2^c)$ such that
$|\beta_1-\beta_1^c|+p|\beta_2-\beta_2^c|<\delta$. For every $u$ in
$[0,1]$, we then have $\llvert l'(u; \beta_1, \beta_2)-l'(u; \beta_1^c,
\beta_2^c)\rrvert<
\delta$. It follows that the zeros $u^*(\beta_1, \beta_2)$ ($u_1^*$
and $u_2^*$ if inside the $V$-shaped
region) must satisfy $\llvert u^*-\frac{p-1}{p}\rrvert<\varepsilon$, which
easily implies the continuity of
$ \frac{\partial}{\partial\beta_1}\psi_\infty$ and
$ \frac{\partial}{\partial\beta_2}\psi_\infty$ at $(\beta
_1^c, \beta_2^c)$.
To see that the transition at the critical point is
second-order, we check the second derivatives of $\psi_\infty$ in its
neighborhood. Off the phase transition
curve,
%
\begin{eqnarray}
\lim_{n\to\infty}\frac{\partial^2}{\partial
\beta_1^2}\psi_n&=&
\frac{\partial^2}{\partial
\beta_1^2}\psi_{\infty}=\frac{\partial}{\partial\beta_1}u^* =-\frac
{1}{l''(u^*)},
%
\\
\lim_{n\to\infty}\frac{\partial^2}{\partial\beta_1 \,\partial
\beta_2}\psi_n&=&
\frac{\partial^2}{\partial\beta_1\,
\partial\beta_2}\psi_{\infty}=\frac{\partial}{\partial
\beta_1}\bigl(u^*
\bigr)^p =-\frac{p(u^*)^{p-1}}{l''(u^*)},
\\
%
\lim_{n\to\infty}\frac{\partial^2}{\partial
\beta_2^2}\psi_n&=&
\frac{\partial^2}{\partial
\beta_2^2}\psi_{\infty}=\frac{\partial}{\partial\beta_2}\bigl(u^*
\bigr)^p =-\frac{(p(u^*)^{p-1})^2}{l''(u^*)}.
\end{eqnarray}
But as was explained in the proof of Proposition \ref{max}, $l''(u^*)$
converges to zero as $(\beta_1, \beta_2)$ approaches $(\beta_1^c,
\beta_2^c)$; the desired singularity is thus justified.
\end{pf*}

\section{Summary}
Much of the literature on phase transitions in exponential random
graph models uses techniques such as mean-field approximations,
which are mathematically uncontrolled. As such they have been
useful in discovering interesting behavior, but they can be
misleading in detail. For instance, although phase transitions
have been discovered in this way for the two-star ($H_2$ a
two-star) \cite{PN1} and edge-triangle ($H_2$ a triangle) models
\cite{PN2}, the approximation leads to an error in the qualitative
nature of the transition, attributing phase coexistence to the
full $V$-shaped region of Figure \ref{V-shape} rather than just the
curve $ \beta_2=q(\beta_1)$; in other words it does
not distinguish the local maxima in the region from the global
maxima.

Chatterjee and Diaconis \cite{CD} gave the first rigorous proof of
singular behavior in an exponential random graph model, the
edge-triangle model. Our paper is an extension
of this important first step; besides extending the models and
parameters under control we have provided a mathematical framework
of ``phases'' which we hope will be useful in motivating future
mathematical work in this subject. Our results show that all
models with ``attraction'' ($ \beta_2 > 0$) exhibit
a transition qualitatively like the gas/liquid transition: a first
order transition corresponding to a discontinuity in density, with a
second order critical point. In Theorem
7.1 of \cite{CD} Chatterjee and Diaconis suggest that, quite generally,
models with repulsion exhibit a transition
qualitatively like the solid/fluid transition, in
which one phase has nontrivial structure, as distinguished from the
``disordered''
Erd\H{o}s--R\'enyi graphs, which have independent edges. We have not yet
been able to extend our results to this regime, except for
$p$-star models with ``repulsion''
($ \beta_2 < 0$) which we prove do not exhibit a transition.
It is an important open problem to determine the qualitative behavior of
models based on an $H_2$ with chromatic number above 2.

\section*{Acknowledgements}

It is a pleasure to acknowledge useful discussions with Persi Diaconis
and Sourav Chatterjee introducing us to the subject of random graphs
and giving us early access to, and explanations of, their preprint
\cite{CD}, as well as suggestions for improving early drafts of this
paper.



\printaddresses

\end{document}